\documentclass[11pt, a4paper]{amsart}

\usepackage[utf8]{inputenc}
\usepackage[T1]{fontenc}
\usepackage[english]{babel}
\usepackage{amsmath}
\usepackage{amsfonts}
\usepackage{lpic}
\usepackage{ae}
\usepackage{units}
\usepackage{icomma}
\usepackage{color}
\usepackage{graphicx}
\usepackage{bbm}
\usepackage{amsthm}
\usepackage{amssymb}
\usepackage{cancel}
\usepackage{fullpage}
\usepackage{upgreek}
\usepackage{algorithm}
\usepackage{algorithmic}
\usepackage{type1cm}
\usepackage{eso-pic}
\usepackage[breaklinks,pdfpagelabels=false]{hyperref}

\renewcommand{\epsilon}{\varepsilon}

\newtheorem{theorem}{Theorem}[section]
\newtheorem{lemma}[theorem]{Lemma}

\newtheorem{conjecture}[theorem]{Conjecture}
\theoremstyle{definition}

\numberwithin{equation}{section}
\numberwithin{theorem}{section}

\begin{document}

\title[The Hegselmann-Krause dynamics for continuous agents and a regular opinion function do not always lead to consensus] 
{The Hegselmann-Krause dynamics for continuous agents and a regular opinion function do not always lead to consensus}

\author{Edvin Wedin$^{1,2}$ \and Peter Hegarty$^{1,2}$} 
\address{$^1$Mathematical Sciences, Chalmers, 41296 Gothenburg, Sweden} 
\address{$^2$Mathematical Sciences, University of Gothenburg,  41296 Gothenburg, Sweden} 
\email{edvinw@student.chalmers.se}

\email{hegarty@chalmers.se}




\date{\today}

\begin{abstract} 
We present an example of a regular opinion function which, as it evolves in accordance with the discrete-time Hegselmann-Krause bounded confidence dynamics, always retains opinions which are separated by more than two. This confirms a conjecture of Blondel, Hendrickx and Tsitsiklis.
\end{abstract}

\maketitle



\setcounter{section}{0} 

\setcounter{equation}{0} 

\section{Introduction}\label{sect:intro}

Recent years have seen an explosion in both the amount of data available to scientists of all persuasions and the computing power necessary to run simulations of mathematical models. There is a rapidly expanding vista for the application of mathematics to the life and social sciences. One major theme of this effort is {\em emergence}, the name given to the idea that patterns in the collective behaviour of large groups of interacting agents may be explicable even if each individual is assumed to obey only rules which are both simple and {\em local}, the latter meaning that each agent is only influenced by its close neighbours, in some appropriate metric. This idea is mathematically appealing, as it suggests a preference for deducing interesting theorems from simple hypotheses. Hence, even if the relevance of any particular mathematical model to ``reality'' may be a difficult and controversial issue, there is at least the promise of some beautiful new mathematical results and conjectures resulting from these efforts.  

The field of {\em opinion dynamics} is concerned with how human 
agents modify their opinions on social issues as a result of the influence of others. This paper is a contribution to the study of a particularly elegant and well-known mathematical model, the bounded confidence model of Hegselmann and Krause \cite{HK}, or simply the {\em HK-model} for brevity.  In the simplest formulation of the model, we have a finite number, say $N$, of agents, indexed by the integers $1,2,\dots,N$. The opinion of agent $i$ is represented by a real number $x(i)$, where the convention is that $x(i) \leq x(j)$ whenever $i \leq j$. The dynamics are as follows: There is a fixed parameter $r > 0$ such that, after each unit of time, every individual replaces their current opinion by the average of those which currently lie within distance $r$ of themselves. This is summarised by the formula
\begin{equation}\label{eq:update}
x_{t+1}(i) = \frac{1}{|\mathcal{N}_{t}(i)|} \sum_{j \in \mathcal{N}_{t}(i)} x_{t}(j),
\end{equation}
where $\mathcal{N}_{t}(i) = \{j : |x_{t}(j) - x_{t}(i)| \leq r \}$. As the dynamics is obviously unaffected by rescaling all opinions and the confidence bound $r$ by a common factor, we can assume without loss of generality that $r = 1$. 
\par Note that the HK-model seems to implicitly assume that each agent is aware of the opinions of all other agents, even if he chooses to ignore most of them when modifying his view. Hence the agents do not really obey ``local'' rules in the sense described earlier. In one sense, this is a matter of interpretation. For example, a conservatively inclined UK citizen may switch the channel whenever a member of the Labour party is giving an interview, or my keep watching but shake his head and mutter under his breath. In other words, the agent adopts strategies which both filter out unwelcome opinions and prevent him from being aware of them in the first place. On the other hand, the HK-model clearly assumes that an agent is aware of all opinions within his current confidence range. There are no restrictions imposed by, for example, geography, which prevent certain agents from sharing opinions a priori. Other important features of the model are that it is fully deterministic and that all agents act simultaneously. Hence the model differs in important respects from other famous models of opinion dynamics such as classical voter models \cite{V} or the Deffuant-Weisbuch model \cite{DNAW}. 

The update rule (\ref{eq:update}) is certainly simple to formulate, though the simplicity is deceptive. Associated to a given configuration $(x(1),\dots,x(N))$ of opinions is a {\em receptivity graph} $G$, whose nodes are the $N$ agents and where an edge is placed between agents $i$ and $j$ whenever 
$|x(i) - x(j)| \leq 1$. The transition in the configuration from time $t$ to time $t+1$ is determined by this graph at time $t$. However, it is clear from (\ref{eq:update}) that the graph will in general vary with time. In algebraic terms, the dynamics are governed by a time-dependent stochastic matrix. This time dependence is the basic reason why many beautiful conjectures about the HK-model remain unresolved, as we shall now explain.        
\\
\par We begin with the necessary notation and terminology. The state space for a system of $N$ agents obeying the HK-dynamics is the set of non-decreasing functions $x: \{1,2,\dots,N\} \rightarrow \mathbb{R}$, equivalently, the set of vectors $(x(1),\cdots,x(N)) \in \mathbb{R}^{N}$ such that $x(i) \leq x(j)$ whenever $i \leq j$. 
An {\em equilibrium} state is one such that $|x(i) - x(j)| > 1$ whenever $x(i) \neq x(j)$. Clearly, once an equilibrium state is reached, then the opinion of every agent will be frozen for all future time. It is also easy to see that the converse holds: if $x_{t+1}(i) = x_t (i)$ for all $i$, then $x_t$ must be an equilibrium state. Any set of agents sharing a common opinion are referred to as a {\em cluster}. By a slight abuse of terminology, the term ``cluster'' may refer either to the set of agents with a certain opinion or the real number representing that opinion. Hence, a HK-system is in equilibrium if and only if no two clusters are within unit distance of each other. The simplest kind of equilibrium state is a {\em consensus}, in which there is only one cluster. Given a cluster $c \in \mathbb{R}$, its weight $w(c)$ is the number of agents sharing opinion $c$. A {\em stable equilibrium} is one in which, for any two clusters $a$ and $b$,  
\begin{equation}\label{eq:stable}
|b-a| \geq 1 + \frac{\min \{w(a), w(b) \}}{\max \{w(a), w(b) \}}.
\end{equation}
This last notion was introduced in \cite{BHT1}, which is the paper that directly inspired the present work. 

The two fundamental facts about the HK-model are the following: 
\\
\\
(A) Any initial state will evolve to equilibrium within a finite time.
\\
(B) Even if the receptivity graph is initially connected, the subsequent equilibrium state need not be a consensus.  
\\
\par Fact (A) seems to have been rediscovered several times over and there are a number of different proofs in the literature. Indeed, the same fact has been proven for a wide class of models of which HK is just one particularly simple example, see \cite{C}. Some of the known proofs of (A) give effective bounds for the time taken to reach equilibrium, as a function of the number $N$ of agents only. The best-known bound is $O(N^3)$, which was proven independently in \cite{BBCN} and \cite{MT}. It is speculated{\footnote{We purposely do not use the word ``conjectured'' here, as we have not seen this hypothesis explicitly stated as a conjecture anywhere in the literature. On the other hand, several authors refer to the example of equally spaced initial opinions, and we have not seen anyone suggest that there may be worse cases than this one.}}, however, that equilibrium is always reached within $O(N)$ steps, and that the worst-case scenario is given by the initial state{\footnote{Simulations suggest that, for large $N$, the initial state $\mathcal{E}_N$ reaches equilibrium in about $cN$ steps, where $c$ is a constant slightly greater than $0.8$.}} $\mathcal{E}_N = (1,2,\dots,N)$. This is an important open problem in the field. 

Regarding (B), it is easy to see that consensus may not be achieved if the initial distribution of opinions is very uneven. For example, suppose we have $100$ agents and the initial state is 
\begin{equation}\label{eq:uneven}
x_0 (i) = \left\{ \begin{array}{lr} -1, & 1 \leq i \leq 98, \\ 0, & i = 99, \\ +1, & i = 100. \end{array} \right.
\end{equation} 
At $t = 1$, the opinion of agent 99 will be pulled very close to $-1$, while agent 100 will only modify his opinion to $x_{1}(100) = 1/2$. Thus, agent 100 will now be isolated from everyone else and will form a cluster by himself in the equilibrium configuration. What is more interesting is that consensus may not emerge even when there is no such unevenness in the initial configuration. The simplest example is the initial state $\mathcal{E}_6$. A direct computation shows that the resulting equilibrium consists of clusters at $\frac{4613}{1728}$ and $\frac{7483}{1728}$, each of weight three. At this point it seems natural to ask what a ``typical'' equilibrium state looks like. In order to make this question precise, let us fix a parameter $L$ and suppose that the initial opinions $x_1(0),\dots,x_N(0)$ are chosen independently and uniformly at random from the interval $[0,\,L]$. The following two conjectures are supported by overwhelming numerical evidence:

\begin{conjecture}\label{conj:critical}
With opinions chosen initially as just described, let $p_{L,N}$ denote the probability that the resulting equilibrium is a consensus. Then there exists a critical value $L_c$\footnote{In \cite{F}, simulations are presented which suggest that $L_C$ is close to 5.} such that, as $N \rightarrow \infty$, $p_{L,N} \rightarrow 1$ whenever $L < L_c$ and $p_{L,N} \rightarrow 0$ whenever $L > L_c$. 
\end{conjecture}   

\begin{conjecture}\label{conj:2R}
With opinions chosen initially as described above, let $q_{L,N}$ denote the probability that the resulting equilibrium is stable, in the sense of (\ref{eq:stable}). Then for any fixed $L$, $q_{L,N} \rightarrow 1$ as $N \rightarrow \infty$.
\end{conjecture}

We have not seen Conjecture \ref{conj:critical} stated explicitly anywhere, though it is implicit in the statements of many different authors. Conjecture \ref{conj:2R} is a special case of Conjecture 1 in \cite{BHT2}. They conjecture that the equilibrium state is almost surely stable under the weaker assumption that the initial opinions are chosen independently from any continuous and bounded pdf on $[0,\,L]$ with connected support, and not just the uniform distribution. Indeed, it is expected that, when the initial distribution is uniform, then the clusters at equilibrium will typically have about the same weight and hence, if (\ref{eq:stable}) holds, will typically be separated by at least two. This hypothesis is referred to in the literature as the {\em $2r$ conjecture}. We are not aware of anyone having turned this hypothesis into a precise conjecture, however. The reason for this is probably that, at least as far as can be told from simulations to date, the distribution of cluster sizes arising from a uniform initial distribution of opinions appears to be quite subtle. 

In an attempt to better understand the behaviour of the HK-model for a large number of agents, Blondel, Hendrickx and Tsitsiklis introduced \cite{BHT1} a {\em continuous agent} version of the model. Here the uncountably many agents are indexed by numbers in the closed interval $[0,\,1]$ and the state space consists of non-decreasing, bounded functions $x: [0,\,1] \rightarrow \mathbb{R}$. The analogue of (\ref{eq:update}) 
is
\begin{equation}\label{eq:contupdate}
x_{t+1}(\alpha) = \frac{1}{\mu(\mathcal{N}_{t}(\alpha))} \int_{\mathcal{N}_{t}(\alpha)} x_{t} (\beta) \, d\beta,
\end{equation}
where $\mathcal{N}_{t}(\alpha) = \{ \beta : |x_{t}(\beta) - x_{t}(\alpha)| \leq 1\}$ and $\mu$ denotes Lebesgue measure. Of course, for (\ref{eq:contupdate}) to even make sense we must assume that the state space contains only Lebesgue measurable functions. An {\em equilibrium state} in this setting is a 
function attaining only finitely many values, such that the difference between any two such values exceeds one whenever both are attained on sets of positive measure. Stable equilibrium can be defined as in (\ref{eq:stable}), where now $w(c) = \mu(x^{-1}(c))$, and the inequality is required to hold whenever both clusters have positive weight. In particular, consensus means a constant function, whereas any equilibrium state which is not a consensus is represented by a discontinuous function. It is natural to assume that the initial state $x_0$ is continuous, however. Intuitively, $x_{0}$ should be injective and $x_{0}^{-1}$ should correspond to the cdf in Conjecture \ref{conj:2R}. This suggests restricting attention to initial states which are $C^1$. On the other hand, there are simple examples where $x_0 \in C^1$ but $x_1$ has corners. For example, suppose $x_{0}(\alpha) = 3\alpha$. An easy computation yields
\begin{equation}\label{eq:corner}
x_{1}(\alpha) = \left\{ \begin{array}{lr} \frac{3\alpha + 1}{2}, & 0 \leq \alpha \leq \frac{1}{3}, \\ 3\alpha, & \frac{1}{3} \leq \alpha \leq \frac{2}{3}, \\ \frac{3\alpha}{2} + 1, & \frac{2}{3} \leq \alpha \leq 1. \end{array} \right.
\end{equation}
Hence $x_1$ is not differentiable at $\alpha = \frac{1}{3}$ and $\alpha = \frac{2}{3}$. We shall assume henceforth that the initial state is {\em regular}, by which we mean that it is almost everywhere $C^1$, with strictly positive lower and upper bounds on its derivative where it ex1ists. This is a slight strengthening of the notion of regularity as defined in \cite{BHT2}. 
\par In contrast to the discrete case, it is not clear whether any initial state will reach equilibrium in finite time. Indeed, since any continuous function 
will obviously remain so when updated according to (\ref{eq:contupdate}), no regular initial state can reach a non-consensus equilibrium in finite time. In \cite{BHT1}, it is conjectured that a regular initial state $x_0$ will converge almost everywhere to a stable equilibrium, that is: there is a stable equilibrium $x_{\infty}$ such that, for each $\varepsilon > 0$ there is a $T_{\varepsilon} > 0$ such that $\mu (\{\alpha: |x_{t}(\alpha) - x_{\infty}(\alpha)| > \varepsilon \}) < \varepsilon$ for all $t > T_{\varepsilon}$. They prove a weaker statement in \cite{BHT2}, but this fundamental conjecture about the continuous agent model remains open. Note, however, that Hendrickx and Olshevsky have recently \cite{HO} proven a corresponding conjecture for another variation on the model where time is also treated as continuous.
\par Thus, for the continuous agent model, a regular initial state will reach equilibrium in finite time if and only if that equilibrium is a consensus -- more precisely, $x_{t+1}$ will be constant if and only if $x_{t}(1) - x_{t}(0) \leq 1$.
This brings us to perhaps the most curious aspect of the continuous agent model, namely it is not obvious that there is any regular initial state which does not reach consensus. The existence of such states was conjectured in \cite{BHT1}, but they could give no example with proof. More precisely, Conjecture 3 of \cite{BHT1} postulates the existence of a regular $x_0$ such that 
\begin{equation}\label{eq:overtwo}
x_{t}(1) - x_{t}(0) \geq 2 \;\; {\hbox{for all $t$}}.
\end{equation}
The motivation for this restriction is that they could prove that, if (\ref{eq:overtwo}) holds, then $x_t$ is regular for all $t$. Our contribution here will be to prove this conjecture:

\begin{theorem}\label{theorem:main}
There exists a regular function $x_0 : [0,\,1] \rightarrow \mathbb{R}$ such that, if the sequence $(x_t)_{t \in \mathbb{N}}$ is defined according to (\ref{eq:contupdate}), then $x_{t}(1) - x_{t}(0) > 2$ for all $t$.
\end{theorem}

Section \ref{sect:proof} contains a proof of this result and Section \ref{sect:discuss} contains a discussion of some open problems. 

\setcounter{equation}{0}

\section{Proof of Main Theorem}\label{sect:proof}

The opinion function to be described below will converge pointwise to a non-regular stable state with 3 clusters of positive weight, and the construction can be extended to allow convergence to (at least) any odd number of such clusters.

\begin{figure}[h]
  \begin{center}
    \includegraphics[scale=0.6]{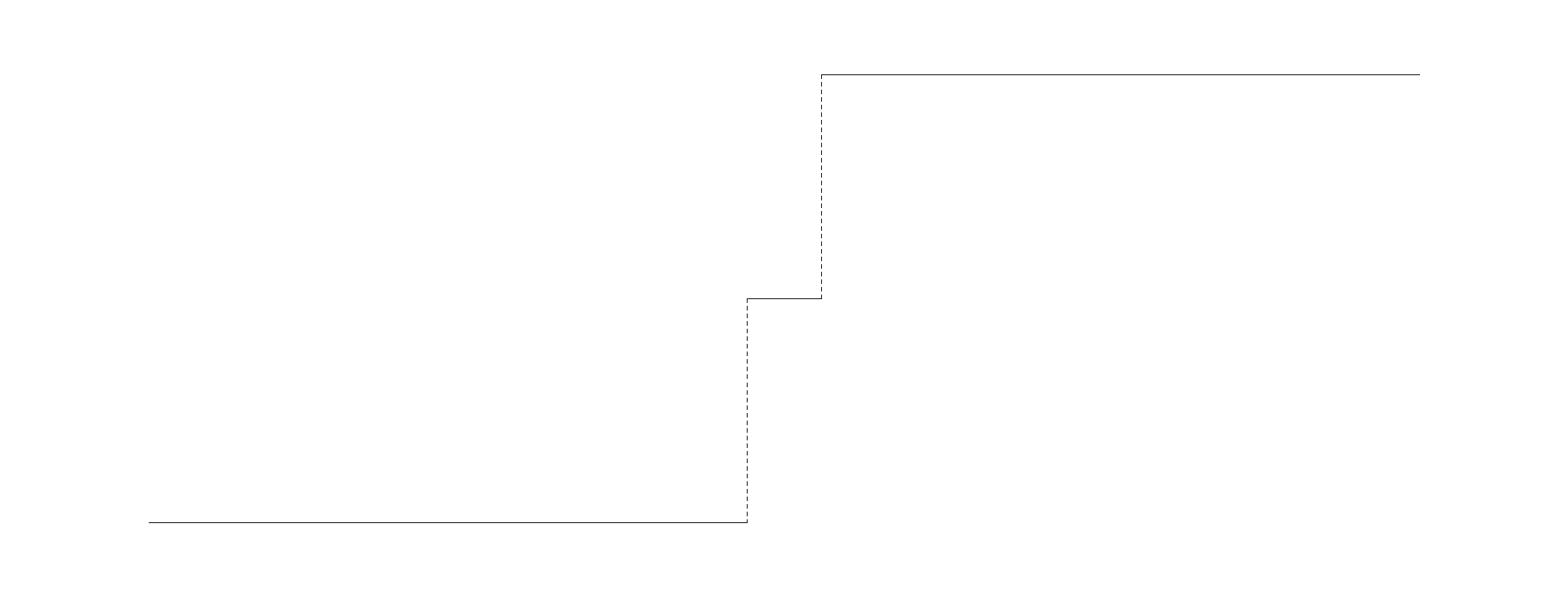}
    \caption{The shape to which the opinion functions will converge}
    \label{Skurva}
  \end{center}
\end{figure}


Let $I=[0,\,1]$.  To construct our initial state $x_0$, we first partition the set $I$ into closed intervals $A$, $B$, $C$, $D$ and $E$, each overlapping the next in exactly one point. For a small $\varepsilon$ that will be specified later, we also require $|B|=|D|=\varepsilon^2 |C|$ and $|C|= \varepsilon^2 |A|=\varepsilon^2 |E|$,  where $|\cdot|$ denotes the length of an interval, so that the endpoints of the intervals lie symmetrically around the centre $c = \frac{1}{2}$ of $C$. We also define $\mathcal{A} = [0,\, \varepsilon]$, $\mathcal{B} = [\varepsilon,\, d+\varepsilon]$ and $\mathcal{C} = [d+\varepsilon,\, d+2\varepsilon]$ for some $d \in (1,\, 2)$ to be subsets of the opinion space. How small $\varepsilon$ needs to be will depend, among other things, on the choice of $d$, so to save us some bookkeeping, we will fix $d=\frac{3}{2}$. These sets are fixed and do not depend on time.

There is some freedom in the construction of the initial state, and in some places conditions will be given along with an example. We will often present an example without explicit bounds on how much it can be varied.

First of all, the initial opinion function $x_0$ is defined to be anti-symmetric about $c$. By this is meant that if coordinates were to be chosen so that $c=0$ and $x_0(c)=0$, then $x_0$ would be odd. From this symmetry, it easily follows that $x_t(c)=x_0(c)$ for all time steps $t$, and that the anti-symmetry remains. 

Second, $x_0$ is defined so that $x_0(A)=\mathcal{A}$, $x_0(B)=\mathcal{B}$ and $x_0(C)=\mathcal{C}$. This will give $x_0$ the shape of a ``double S'', with a small plateau in the middle and two long tails, see Figure 
\ref{Skurva}
. Not only should $x_0$ stay within these ``boxes'', illustrated in Figure 
\ref{box1}
, but we require also that the derivative fulfil that $x_0'(\alpha)|_A \leq e_0$ and  $x_0'(\alpha)|_B \geq s_0$ for suitable constants $e_0$ and $s_0$. In this example, we will take $x_0$ to be linear on all intervals $A$, $B$, $C$, $D$ and $E$, and $e_0$ and $s_0$ to equal the derivatives on $A$ and $B$, respectively, so that $e_0=\frac{\varepsilon}{|A|}$ and $s_0=\frac{d}{|B|}$. 

\begin{figure}
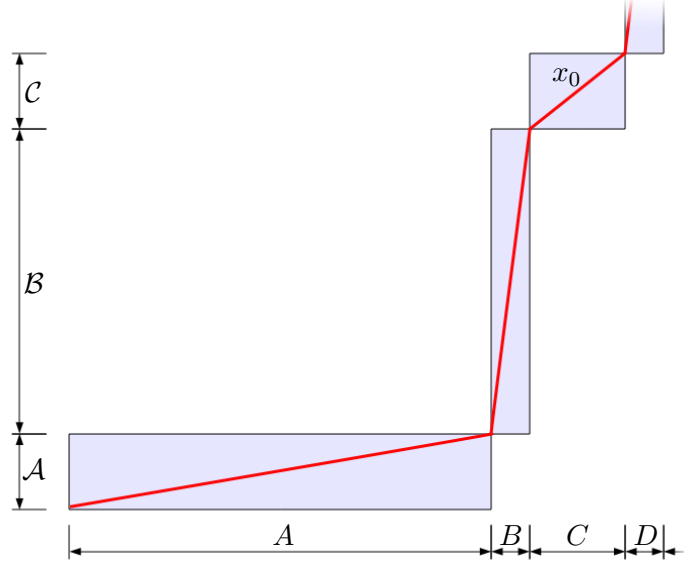

  \begin{center}
    \begin{lpic}[clear]{lador2gimpad(0.53)}
      \lbl[b]{71,11;$A$}
      \lbl[b]{127.5,11;$B$}
      \lbl[b]{144,11;$C$}
      \lbl[b]{161,11;$D$}
      \lbl[b]{9,27.5;$\mathcal{A}$}
      \lbl[b]{9,74;$\mathcal{B}$}
      \lbl[b]{9,122;$\mathcal{C}$}
      \lbl[br]{145,126;$x_0$}
    \end{lpic}
    \caption{Some time invariant subspaces of $I$ and the opinion space for a piecewise linear initial opinion function.}
    \label{box1}
  \end{center}
\end{figure}


To prove that, from the initial state $x_0$, condition (\ref{eq:overtwo}) will be satisfied, a sort of induction will be used. For every time step, the average opinion for the agents in $A$ will increase, and we will prove that this increase is at most linear in the measure of the set of agents with opinion in a certain subset of $\mathcal{B}$. We will then show that this measure will shrink quickly enough for the sum of the increments in $A$ to converge, and thereby show that there will always be agents in $A$ with opinion less than $\varepsilon$, provided $\varepsilon$ is chosen small enough. Because of symmetry, this will imply the existence of agents in $E$ with opinions greater than $2d+2\varepsilon$, and since $d>1$ this will complete the proof.

\begin{figure}
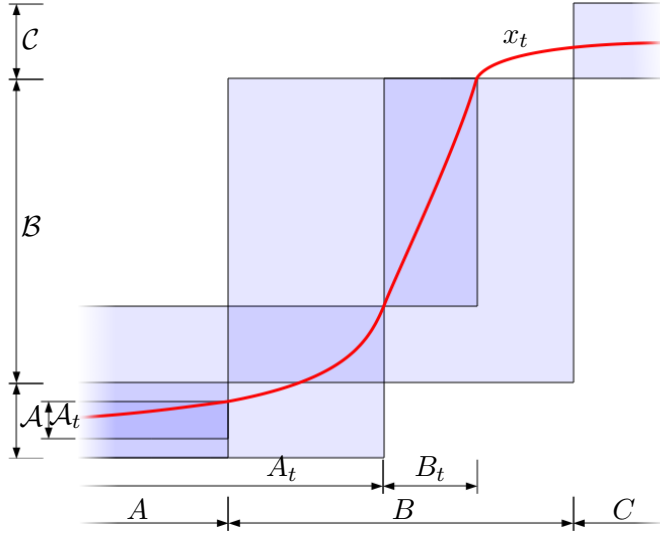

\begin{center}
    \begin{lpic}[clear]{ladorgimpad(0.5)}
      \lbl[b]{114,20;$B_t$}
      \lbl[b]{75,20;$A_t$}
      \lbl[b]{107,10;$B$}
      \lbl[b]{37,10;$A$}
      \lbl[b]{165,10;$C$}
      \lbl[b]{9,34.5;$\mathcal{A}$}
      \lbl[b]{18,34;$\mathcal{A}_t$}
      \lbl[b]{9,85;$\mathcal{B}$}
      \lbl[b]{9,134.5;$\mathcal{C}$}
      \lbl[br]{140,135;$x_t$}
    \end{lpic}
    \caption{Overview of the subsets of the agent space and the opinion space at time $t$.
      }
    \label{box2}
  \end{center}
\end{figure}


Let $A_t= x_t^{-1}([0,\, 2 \varepsilon])$ and $B_t= x_t^{-1}([2\varepsilon,\, \varepsilon+d])$ be the sets of agents with opinions in $[0,\, 2\varepsilon]$ and $[2\varepsilon,\, \varepsilon+d] $, respectively, at time $t$, see Figure \ref{box2}
. The reason for using $2\varepsilon$ instead of just $\varepsilon$ will become clear in the proof of Lemma \ref{lem:växder} below. We also let $\mathcal{A}_t= x_t(A)$. We will let $\bar{A_t}$ denote the average opinion on $A_t$ at time $t$, and a vertical bar with a subscript will denote the restriction of a function to the set in the subscript.

The precise assumptions on the function at time $t$ are the following:
\begin{description}
\item[I] $A \subseteq A_t$.\nopagebreak[4]
\item[II] $B \supseteq B_t$.\nopagebreak[4]
\item[III] $\mathcal{A}_t\subseteq \mathcal{A}$.\nopagebreak[4]
\item[IV] There is an $e_t$ such that $x_t'|_A \leq e_t$.\nopagebreak[4]
\item[V] There is an $s_t$ such that $x_t'|_{B_t} \geq s_t$.  \nopagebreak[4]
\item[VI] $\varepsilon - \bar{A_t} > 2\varepsilon^2$.\nopagebreak[4] 
\end{description}
Note that with, from our definition of $x_0$, it is clear that all but the last assumption hold for $t=0$. For the last assumption to hold $\varepsilon$ must be small enough, and it is easy to check that $\varepsilon\leq \frac{1}{4}$ is enough at $t=0$.

Under these assumptions we will prove the following lemmas:

\begin{lemma}\label{lem:nästab}
  Only the agents who at time $t$ have opinions in  $\mathcal{A}_t+1=\{\alpha+1:\alpha \in \mathcal{A}_t\}$ may end up in $B_{t+1}$, and  $\mathcal{A}_{t+1}\subset \mathcal{A}$. An immediate consequence of this is that   $A_{t+1}\supseteq A_t$ and $B_{t+1}\subseteq B_t$.
\end{lemma}

\begin{proof}[Proof of Lemma \ref{lem:nästab}]
    We simply examine the two extreme agents in $x_t^{-1}(\mathcal{A}_t+1)$ to see that these will have opinions on opposite sides of $\mathcal{B}$ at time $t+1$. Let
    \begin{equation*}
      \beta_t = x_t^{-1}(\min ~ \mathcal{A}_t+1),~~~~
      \gamma_t = x_t^{-1}(\max~ \mathcal{A}_t+1).
    \end{equation*}

    Firstly, we observe that the agent $\gamma_t$ can see all agents in $B$ and $C$, some agents in $D$, and one agent in $A$. For the function 
    \begin{equation*}
      \tilde{x}_t(\alpha)=
      \begin{cases}
        \frac{3}{2}\varepsilon+d, & \text{if } \alpha \in C,\\
        0, & \text{else}
      \end{cases}
    \end{equation*}
    we thus have that
    \begin{equation}
      \label{eq:lillaeps1}
      x_{t+1}(\gamma_t) \geq
      \frac{1}{|\mathcal{N}_t(\gamma_t)|} \int_{\mathcal{N}_t(\gamma_t)}\tilde{x}_t(\alpha) d\alpha \geq
      \frac{|C|(\frac{3}{2}\varepsilon+d)}{|B|+|C|+|D|} =
      \frac{|C|(\frac{3}{2}\varepsilon+d)}{|C|(2\varepsilon^2+1)} \geq
      \varepsilon+d
    \end{equation}
    if\footnote{\label{fotnot.}To compute this bound we use $d=\frac{3}{2}$} $\varepsilon \leq \frac{1}{4}(\sqrt{13}-3)$, so $x_{t+1}(\gamma_t)\not\in \mathcal{B}$. The second inequality follows precisely from the fact that $\gamma_t$ cannot see any agents in $D$, and only one agent in $A$ that is also in $B$, which implies $\mathcal{N}_t(\gamma_t) \subseteq B\cup C\cup D$.

    Secondly, we observe that the agent $\beta_t$ can see all agents in $A$, $B$ and $C$, and some of the agents in $D$. The function
    \begin{equation*}
      \hat{x}_t(\alpha)=
      \begin{cases}
        \bar{A}_t, & \text{if } \alpha \in A_t, \\
        \varepsilon+d, & \text{if } \alpha \in B_t, \\
        \frac{3}{2} \varepsilon + d, & \text{if } \alpha \in (B \setminus (A_t \cup B_t)) \cup C, \\ 
        2\varepsilon + 2 d, &  \text{if } \alpha \in D \\
      \end{cases}
    \end{equation*}
    has averages greater than or equal to $x_t$ on all the subsets of $I$ that can be seen from $\beta_t$. We can thereby use $\hat{x}_t$ to get the following bound:
    \begin{align}
       x_{t+1}(\beta_t) \leq
      \frac{1}{|\mathcal{N}_t(\beta_t)|}\int_{A \cup B \cup C \cup D} \hat{x}_t(\alpha) d \alpha \leq \nonumber\\
      \leq \frac{|A_t|\bar{A_t}+\varepsilon^4|A_t|(d+\varepsilon)+\varepsilon^2|A_t|(\frac{3}{2}\varepsilon+d)+\varepsilon^4|A_t|(2\varepsilon+2d)}{|A_t|} < \nonumber\\
      < \bar{A}_t + 3 \varepsilon^5 + 3 d \varepsilon^4 +\frac{3}{2} \varepsilon^3 +d\varepsilon^2 <
      \bar{A_t}+2\varepsilon^2 <
      \varepsilon \not\in \mathcal{B}
      \label{lillaeps2}
    \end{align}
for at least\footnotemark[\value{footnote}] $\varepsilon \leq \frac{1}{10}$, where the last inequality is true according to Assumption VI. This also proves that $\mathcal{A}_{t+1} \subseteq \mathcal{A}$.

We have now shown that $x_{t+1}(\beta_t)$ and $x_{t+1}(\gamma_t)$ both lie outside $\mathcal{B}$, and monotonicity now completes the proof.
\end{proof}

\begin{lemma}
  \label{lem:medelökning}
  The increase in the mean opinion from $A_t$ at time $t$ to $A_{t+1}$ at time $t+1$ is at most linear in $|B_t|$. More precisely, $\bar{A}_{t+1}-\bar{A}_t \leq 4\frac{|B_t|}{|A|}$.
\end{lemma}
\begin{proof}[Proof of Lemma \ref{lem:medelökning}]
  Lemma \ref{lem:nästab} tells us that $A_{t+1} \supseteq A_t$, and this allows us to write $A_{t+1}= A_t \sqcup (A_{t+1} \setminus A_t)$, a disjoint union of two non-empty sets. 

  As for the first of these sets, the part of $B_t$ that is visible from $A_t$ has opinions that do not exceed $1+2\varepsilon$, and so the average of $x_{t+1}$ over $A_t$ at time $t+1$  will be
  \begin{align*}
    \overline{(A_t)}_{t+1} \leq
    \frac{|A|\bar{A}_t+|B_t|(1+2\varepsilon)}{|A_t|+|B_t|}<
    \bar{A}_t + \frac{|B_t|(1+2\varepsilon)}{|A_t|}<
    \bar{A}_t + \frac{2|B_t|}{|A|},
  \end{align*}
  where the two last inequalities use that $A\subseteq A_t$, which is true by Assumption I. Note that this bound is simply a bound on the opinion of the rightmost agent in $A_t$ at time $t+1$, that also bounds the whole average because of monotonicity.

 The average of $x_{t+1}$ over $A_{t+1} \setminus A_t$ at time $t+1$ is certainly at most $2\varepsilon$, by definition of the set $A_{t+1}$. We also know from Lemma \ref{lem:nästab} that $A_{t+1} \setminus A_t \subseteq B_t$, and we thereby get the total average
\begin{align*}
  \bar{A}_{t+1}\leq
  \frac{\int_{A_{t+1}}x_{t+1}(\alpha) d\alpha}{|A_{t+1}|} \leq
  \frac{|A_{t}|\overline{(A_t)}_{t+1}+|A_{1+t} \setminus A_t|(2\varepsilon)}{|A_{t}|+|A_{1+t} \setminus A_t|} \leq\\
  \leq
  \overline{(A_t)}_{t+1}+\frac{|B_{t}|(2\varepsilon)}{|A|} \leq
  \bar{A}_t + \frac{2|B_t|}{|A|}+\frac{|B_{t}|(2\varepsilon)}{|A|} \leq 
  \bar{A}_t + \frac{4|B_t|}{|A|},
\end{align*}
   as long as $\varepsilon\leq 1$.
\end{proof}

\begin{lemma}
\label{lem:krympder}
The derivative of $x_{t+1}$ on $A$ is bounded from above: 
\begin{equation*}
  x'_{t+1}|_{A} \leq
  e_{t+1}=
  \frac{2}{|A|}\frac{e_t}{s_t}.
\end{equation*}
Provided $s_t$ is big enough, which we will show is the case, this implies the much weaker statement $e_{t+1}\leq e_t$.
\end{lemma}

\begin{lemma}
  \label{lem:växder}
  The derivative of $x_{t+1}$ on $B_{t+1}$ is bounded from below:
  \begin{equation*}
    x'_{t+1} \geq
    s_{t+1} =
    \frac{\varepsilon}{2|A|}\frac{s_t}{e_t}.
  \end{equation*}
  This gives us a bound on the size of $B_{t+1}$:
  \begin{equation*}
    |B_{t+1}| \leq
    \frac{d}{s_{t+1}}.
  \end{equation*}
\end{lemma}

In the proofs of Lemmas \ref{lem:krympder} and \ref{lem:växder} the following additional lemma will be used:
\begin{lemma}
  \label{lem:derivata}
  Let $x_t$ be a regular opinion function on $I$ such that $x_t(1) - x_t (0) > 2$ and define the functions $u_t,~v_t,~w_t:I \rightarrow I$ as follows:
\[
u_t(\alpha) = 
\begin{cases}
  0, & \text{if $x_t(\alpha)\leq x_t(0)+1$},\\
  x_t^{-1}(x_t(\alpha)-1), & \text{otherwise},
\end{cases}
\]
\[
v_t(\alpha) =
\begin{cases}
  1, & \text{if $x_t(\alpha)\geq x_t(1)-1$},\\
x_t^{-1}(x_t(\alpha)+1), & \text{otherwise},
\end{cases}
\]
\[
w_t(\alpha)= v_t(\alpha)-u_t(\alpha),
\]
Then the updated function $x_{t+1}$ is regular with derivative, where it exists, given by
\begin{equation}
  \label{eq:derivekv}
 x_{t+1}' (\alpha)= \frac{1}{w_t(\alpha)}
 \left[
   u_t'(\alpha)\cdot(1 + x_{t+1}(\alpha)-x_t(\alpha))+
   v_t'(\alpha)\cdot(1 + x_t(\alpha)-x_{t+1}(\alpha))
   \right].
\end{equation}
\end{lemma}

\begin{proof}
That $x_{t+1}$ is regular was proven in Lemma 4 of \cite{BHT1}. Assuming $x_{t+1}$, $u_t$, $v_t$ and $w_t$ are differentiable at $\alpha$, we can compute as follows:

We first use the definition in (\ref{eq:contupdate}) along with the product rule for derivatives to get
\begin{equation}
  \label{eq:summa}
  x_{t+1}^{\prime}(\alpha) =
  \frac{-w_t^{\prime}(\alpha)}{(w_t(\alpha))^2}  \int_{u_t(\alpha)}^{v_t(\alpha)} x_t(\beta) d\beta + \frac{1}{w_t(\alpha)} \frac{d}{d\alpha}   \left( \int_{u_t(\alpha)}^{v_t(\alpha)} x_t(\beta)  d\beta \right).
\end{equation}
The first term simplifies to 
\begin{equation}
\label{eq:term1}
\frac{[u_t^{\prime}(\alpha) - v_t^{\prime}(\alpha)]x_{t+1}(\alpha)}{w_t(\alpha)}
\end{equation}
and, by the chain rule, the second term can be rewritten as
\begin{equation}
\label{eq:term2}
\frac{1}{w_t(\alpha)} \left[ x_t(v_t(\alpha)) \cdot v_t^{\prime}(\alpha) - x_t(u_t(\alpha)) \cdot u_t^{\prime}(\alpha) \right].
\end{equation}
But by definition of the functions $u_t$ and $v_t$, we have 
\begin{equation*}
  \begin{array}{r}
  x_t(v_t(\alpha)) =  \begin{cases}
    x_t(\alpha) + 1, & \text{if $x_t(\alpha) \leq x_t(1) - 1$}, \\ 
    1, & \text{otherwise}, 
  \end{cases}
  \\
  \\
  x_t(u_t(\alpha)) =  \begin{cases}
    x_t(\alpha) - 1, & \text{if $x_t(\alpha) \geq x_t(0) + 1$}, \\ 0, & \text{otherwise}.
  \end{cases}
  \end{array}
\end{equation*}
We would like to substitute the values $x_t(v_t(\alpha)) = x_t(\alpha) + 1$ and $x_t(u_t(\alpha)) = x_t(\alpha) - 1$ into (\ref{eq:term2}). The former doesn't hold when $x_t(\alpha) > x_t(1) - 1$, but in this range $v_t(\alpha) = 1$ so $v_t^{\prime}(\alpha) = 0$, so we can make the substitution anyway and it doesn't matter. A similar reasoning applies to the latter substitution. Hence the right-hand side of (\ref{eq:term2}) simplifies to 
\begin{equation}
  \label{eq:term22}
  \frac{v_t^{\prime}(\alpha) \cdot (x_t(\alpha) + 1) - u_t^{\prime}(\alpha) \cdot (x_t(\alpha) - 1)}{w_t(\alpha)}.
\end{equation}
Substituting (\ref{eq:term1}) and (\ref{eq:term22}) into (\ref{eq:summa}) leads after a little computation to (\ref{eq:derivekv}).
\end{proof}

\begin{proof}[Proof of Lemma \ref{lem:krympder}]
  For agents $\alpha$ in $A$ we have that $u_t(\alpha) = 0$, so $u'_t(\alpha)=0$. Using Lemma \ref{lem:derivata} 
this gives 
  \begin{equation}
    \label{eq:beviskrympder}
    x'_{t+1}(\alpha) = 
    \frac{1+x_t(\alpha)-x_{t+1}(\alpha)}{w_t(\alpha)}v'_t(\alpha)
  \end{equation}
  for  all $\alpha \in A$. From Assumption III we know that $\mathcal{A}_t \subseteq \mathcal{A}$, so we know that all agents in $A$ can see each other and at least some agents in $B_t$, but no agents beyond $B_t$, so $w_t(\alpha)>|A|$.

To get a bound on $v'_t(\alpha)$, we use the definition of $v_t$, the chain rule, and the formula for the derivative of an inverse function:
\begin{equation*}
_t(\alpha)=
  \frac{d}{d\alpha} x_t^{-1}(x_t(\alpha)+1)=
  x'_t(\alpha) \cdot \frac{1}{x_t'(x_t^{-1}(x_t(\alpha)+1))} =
  \frac{x_t'(\alpha)}{x_t'(v_t(\alpha))}.
\end{equation*}
To bound this, first note that $\alpha \in A$ implies $x'_t(\alpha) \leq e_t$. Second, note that since we assume $\mathcal{A}_t\subseteq \mathcal{A}$, it follows that $x_t(v_t(\alpha))=[1,\,1+\varepsilon]$. In particular, $v_t(\alpha)\in B_t$, and thus $x_t'(v_t(\alpha))\geq s_t$. Putting this together results in the bound
\begin{equation}
  \label{eq:vprim}
  v'_t(\alpha) \leq \frac{e_t}{s_t}.
\end{equation}

Finally we observe that $1+x_t(\alpha)-x_{t+1}(\alpha)\leq 2$ holds trivially, and we can now insert this and (\ref{eq:vprim}) into (\ref{eq:beviskrympder}) to obtain
  \begin{equation*}
    x'_{t+1}(\alpha) \leq
    \frac{2}{|A|}\frac{e_t}{s_t},
  \end{equation*}
  as desired.
\end{proof}
\begin{proof}[Proof of Lemma \ref{lem:växder}]
  First observe that since both the  terms within brackets in (\ref{eq:derivekv}) are positive, only one of them is needed to construct a lower bound for the derivative:
  \begin{equation}
    \label{eq:bevisväxder}
    x'_{t+1}(\alpha) \geq
    \frac{1}{w_t(\alpha)}u'_t(\alpha)[1+x_{t+1}(\alpha)-x_t(\alpha)].
  \end{equation}
  We know from Assumption III and symmetry that no agent in $B_t$ can see as far as $E$, so $w_t(\alpha)\leq |A \cup B  \cup C \cup D|<2|A|$. The first statement in Lemma \ref{lem:nästab} and the definition of $B_t$ together assure us that $(1+x_{t+1}(\alpha)-x_{t}(\alpha))\geq \varepsilon$: All the agents in $B_{t+1}$ must have had opinions in $\mathcal{A}_t+1$ at time $t$, or they would have ended up outside $B_{t+1}$ after the update, according to Lemma \ref{lem:nästab}. This is the motivation for using $2\varepsilon$ in the definitions of $A_t$ and $B_t$. It also lets us  use $e_t$ and $s_t$ in a way similar to what was done in the proof of Lemma \ref{lem:krympder} to get that $u'_t(\alpha) = \frac{x_t'(\alpha)}{x_t'(u_t(\alpha))} \geq \frac{s_t}{e_t}$. Applying these inequalities to (\ref{eq:bevisväxder}) gives the result.

  The upper bound on the size of $B_{t+1}$ simply comes from multiplying the inverse of the bound on the derivative with the height of $B_{t+1}$, which  we know is constantly $d-\varepsilon<d$ by construction.
\end{proof}

\begin{proof}[Proof of Theorem \ref{theorem:main}]
To begin with note that, if we were to rescale the inteval $I$ it would not affect any of our arguments. Hence, in order to simplify notation, we rescale so that $|A|=1$.

  We would like Assumptions I-VI to be true for all time steps, for then we would be done. Lemmas \ref{lem:nästab}, \ref{lem:krympder} and \ref{lem:växder} show that Assumptions I-V can be made for time $t+1$ if they hold for time $t$, provided that $\varepsilon$ is small enough. Take  $\varepsilon=\frac{1}{100}$. This is small enough for all the previous arguments to go through,
 and will be small enough for the arguments to follow.

As for  Assumption VI, Lemma \ref{lem:medelökning} tells us that the mean value $\bar{A_t}$ on the increasing sequence of sets $A_t$ depends linearly on the sizes $|B_t|$. From monotonicity and the fact that $A_t \supset A$ by Assumption I, we know that the average on $A$ at time $t$ is less than $\bar{A}_t$, so an upper bound on $\bar{A}_t$ will be enough to finish the proof. 

 To obtain a bound on $\varepsilon$ that guarantees convergence of the values $\bar{A_t}$, we will use Lemmas \ref{lem:krympder} and \ref{lem:växder}, mostly the weaker statement in Lemma \ref{lem:krympder}. At time $0$, assuming $|A|=1$, we have $e_0=\varepsilon$ and $s_0=\frac{d}{\varepsilon^4}$. At time $1$ we have $e_1=\frac{2 \varepsilon^5}{d}$ and $s_1=\frac{d}{2\varepsilon^4}$, and this will be enough to complete the argument as we get the bound
\begin{equation*}
  s_{t+1}=
  \frac{\varepsilon}{2e_t}s_t \geq
  \frac{\varepsilon}{2e_1}s_t =
  \frac{d}{4 \varepsilon^5}s_t ,
\end{equation*}
for $t\geq 1$, and hence, by iteration
\begin{equation*}
  s_{t} \geq
  \left(\frac{d}{4\varepsilon^5}\right)^{t-1}s_1,
\end{equation*}
for $t\geq 1$. This might appear like circular reasoning at a first glance: We show that the variables $s_t$ are growing using the fact that the variables $e_t$ are not, but this in turn depends on the variables $s_t$ being big enough. To see that the argument actually holds, note that the inequalities can be checked one time step at a time, and that the correctness of every step follows from the previous one.

With the bound on $|B_{t+1}|$ from Lemmas \ref{lem:växder} and \ref{lem:medelökning} we get that
\begin{equation*}
  \bar{A}_{t+1}-\bar{A}_t\leq
  4 B_t \leq
  4\frac{d}{s_1}  \left(\frac{4\varepsilon^5}{d}\right)^{t-1}=
  8 \varepsilon^4 \left(\frac{4\varepsilon^5}{d}\right)^{t-1},
\end{equation*}
for  $t\geq 1$. If $t=0$ we instead get that
\begin{equation*}
  \bar{A}_1-\bar{A}_0\leq
  4|B_0|=
  4 \varepsilon^4,
\end{equation*}
which lets us sum up the bounds on the increments and obtain
\begin{equation*}
  \lim_{t\rightarrow \infty} \bar{A}_t \leq
  \frac{\varepsilon}{2}+
  4 \varepsilon^4+
  8\varepsilon^4 \frac{1}{1-\left(\frac{4 \varepsilon^5}{d}\right)}<
  \varepsilon-2\varepsilon^2,
\end{equation*}
where the last inequality holds at least for $\varepsilon=\frac{1}{100}$. This completes the proof.
\end{proof}

\setcounter{equation}{0}

\section{Discussion}\label{sect:discuss}
When thinking about how to construct an example to prove Theorem \ref{theorem:main}, we first considered a "single-S" shape, without the narrow plateau in the middle, but with the height of the narrow strip connecting the two tails still being above two. We could not prove that the updates of such an initial state would also satisfy (\ref{eq:overtwo}), though we suspect this is the case. In fact, what we think happens when the function is updated is that a narrow plateau will form in the middle, thus yielding the "double-S" shape of the function in Section \ref{sect:proof} as an intermediate step in the evolution.

In any case, there should be even simpler examples of regular functions which satisfy (\ref{eq:overtwo}). Indeed, Conjecture \ref{conj:critical} suggests the following corresponding hypothesis for the continuous agent model:

\begin{conjecture}\label{conj:linear}
Let $x_0: [0, \, 1] \rightarrow \mathbb{R}$ be a non-decreasing linear function. Then there is a critical value $L^{*}_{c}$ such that the updates $x_t$ satisfy (\ref{eq:overtwo}) whenever $x_0(1) - x_0(0) > L^{*}_c$, whereas $x_0$ will evolve to consensus when $x_{0}(1) - x_{0}(0) < L^{*}_{c}$. Moreover, $L^{*}_{c} = L_c$, the critical value in Conjecture \ref{conj:critical}.
\end{conjecture}

In fact, we also conjecture there will not be evolution to consensus at the critical value $L^{*}_c$. Intuitively, the reason for this is as follows. For any continuous $x_0$, the ranges $x_t(1) - x_t(0)$ of the updates will be strictly decreasing with $t$ as long as we don't have consensus. The "$2r$ conjecture" suggests that, given a linear $x_0$, there will eventually be consensus if and only if the range of opinions shrinks to strictly below two at some point. Hence, at $L = L^{*}_{c}$, we should converge almost everywhere to an equilibrium consisting of two clusters of equal measure and separated by exactly two.

This leads in turn to another obvious remaining question, namely whether it is possible for a regular initial state to fail to satisfy (\ref{eq:overtwo}) and yet never reach consensus.

\end{document}